\documentclass[12pt]{article}
\usepackage{amsmath,amsfonts,amssymb,graphicx,url,a4}

\def\z{{\mathbb Z}}
\def\N{{\mathbb N}}
\def\n{{\mathbb N}}
\def\Zm{{-\mathbb{N}}}
\def\P{{\mathbb P}}
\def\E{{E}}
\def\1{{\mathbb 1}}
\def\F{{\mathcal{F}}}
\def\G{{\mathcal{G}}}

\newtheorem{theoreme}{Theorem}
\newtheorem{conjecture}{Conjecture}[section]
\newtheorem{lemme}[conjecture]{Lemma}
\newtheorem{proposition}[conjecture]{Proposition}

\newtheorem{remarque}{Remark}
\newtheorem{exemple}{Example}

\newtheorem{definition}{Definition}
\newtheorem{notation}{Notation}

\makeatletter
\def\@yproof[#1]{\@proof{ #1}}
\def\@proof#1{\begin{trivlist}\item[]{\em Proof#1.}}
\newenvironment{proof}{\@ifnextchar[{\@yproof}{\@proof{}
}}{~\end{trivlist}} \makeatother

\title{The filtration of the split-words process}
\author{Ga\"el Ceillier}

\renewcommand{\le}{\leqslant}
\renewcommand{\ge}{\geqslant}

\begin{document}

\maketitle


Soon to be published in Probability
Theory and Related Fields.



\begin{abstract}
M. Smorodinsky and S. Laurent have initiated the study of the filtrations of
split-word processes, in the framework of discrete negative time. For these filtrations,
we show that Laurent's sufficient condition for non standardness is also necessary,
thus yielding a practical standardness criterion. In turn, this criterion enables us to
exhibit a non standard filtration which becomes standard when time is accelerated by
omitting infinitely many instants of time.
\end{abstract}

\section{Introduction}
We shall be interested in filtrations, in the setting of discrete, negative time: given
a probability space $(\Omega,\mathcal{A}, \P)$, a filtration is an increasing family $\F = (\F_n)_{n\le 0}$ of
sub-$\sigma$-fields of $\mathcal{A}$; observe that the time $n$ ranges over all negative integers. (Equivalently,
one could consider decreasing families of $\sigma$-fields indexed by positive integers, known
as reverse filtrations; but we find it more convenient to let time run forward, with
$n+1$ posterior to $n$, at the mild cost of dealing with negative instants.) As discovered
by Vershik~\cite{vershik1995tds}, in this framework very subtle phenomena occur in the vicinity
of time $-\infty$.

For a simple example, suppose that the $\sigma$-field
$\cap_n\F_n$ is degenerate and that, for
each n, $\F_n$ is generated by $\F_{n-1}$ and by some Bernoulli random variable $U_n$ which is independent
of $\F_{n-1}$ and uniformly distributed on the $2$-set $\{0, 1\}$. Under these hypotheses, it
may happen that $\F$ contains more information than the natural filtration of the
Bernoulli process $ U = (U_n)_{n\le 0}$ (this is similar to weak solutions in SDEs); but
something more surprising is also possible: that $\F$ is not generated by any Bernoulli
process whatsoever. Such filtrations have been called non standard by Vershik, who
has given in~\cite{vershik1995tds} a necessary and sufficient criterion for standardness, and several
examples of non standard filtrations.
The rigorous definition of a standard filtration will be recalled later, in section~\ref{S3-split}.

All filtrations considered in this study have an additional property: for each n,
$\F_n$ is generated by $\F_{n-1}$ and by some random variable $U_n$ which is independent from $\F_{n-1}$ and
uniformly distributed on some finite set with $r_n$ elements. Such a filtration is called
$(r_n)$-adic. For these filtrations, as shown by Vershik~\cite{vershik1995tds}, standardness turns out to
be tantamount to a simpler, much more intuitive property: an $(r_n)$-adic filtration
$\F$ is standard if and only if $\F$ is of product type, that is, $\F$ is the natural filtration
of some process $V = (V_n)_{n\le 0}$ where the $V_n$ are independent random variables (in this case, it is
easy to see that the process $V$ can be chosen with the same law as $U$.) So, at first
reading, `standard' can be replaced with `of product type' in this introduction.

When time is accelerated by extracting a subsequence, that is, when ($\F_n)_{n\in\z^-}$
is replaced with ($\F_n)_{n\in Q}$ where $Q$ is some infinite subset of the time-axis $\z^-$, a
standard filtration always remains standard, but a non standard one may become
standard (or not). Examples of this phenomenon were first studied by Vershik in
the framework of ergodic theory, and then, in a probabilistic setting, by Laurent~\cite{laurent2004ftd}.
Lacunary isomorphism theorem~\cite{emery2001vershik} states that, from any filtration $(\F_n)_{n\in\z^-}$
(such that $\F_0$ is essentially separable), there exists $Q \subset \z^-$ such that ($\F_n)_{n\in Q}$ is standard.

By varying the parameters in an example initially due to Vershik~\cite{vershik1995tds} and
later modified (in the dyadic case) by M. Smorodinsky~\cite{smorodinsky1998pns}, Laurent~\cite{laurent2004ftd} has
described a family of filtrations, the split-word filtrations; he has shown some of
them (the fastest ones) to be standard, and some other ones (the slowest ones) to
be non standard; but an intermediate class was left undecided. We continue his
study, and show that all these intermediate filtrations are in fact standard. This
yields an easily verifiable necessary and sufficient condition for a split-word filtration
to be standard.

As the family of split-word filtrations is stable by extracting subsequences, this
criterion makes it simple to observe on these examples the transition from nonstandardness
to standardness when time is accelerated. We find that this transition
is, in some sense, sharp: in Example~\ref{example_non-standard}, we exhibit a non standard filtration $\F$ such
that, for every infinite subset $Q$ of $\z^-$ with infinite complementary, the corresponding
extracted filtration ($\F_n)_{n\in Q}$ is standard. This $\F$ is as close to being standard as
possible, for, if $Q \subset \z^-$ is cofinite and if $\G$ is any filtration, the extracted filtration
$(\G_n)_{n\in Q}$ clearly has the same asymptotic properties (standardness, product type,
etc.) as $\G$. To our knowledge, in the earlier literature, the best result in this direction
was the existence of a non standard filtration $\F$ such that $(\F_{2n})_{n\le 0}$ is standard
(examples are given by Vershik~\cite{vershik1995tds}, Gorbulsky~\cite{gorbulsky1999one} and Tsirelson in an unpublished paper).

In this paper we study the filtrations of split-word processes. These processes are inspired by examples given by Vershik~\cite[example~2,3,4]{vershik1995tds} and have been introduced and studied in terms of probability theory by Smorodinsky~\cite{smorodinsky1998pns} in the dyadic case and by Laurent~\cite{laurent2004ftd} in the general case. 

The distribution of a split-word process depends on an alphabet $A$ of size $N \ge 2$ and on a sequence of positive integers  $(\ell_{n})_{n \le 0}$ such that $\ell_0=1$ and, for every $n\le0$, the ratio $r_n=\ell_{n-1}/\ell_{n}$ is an integer $r_n\ge2$.
The sequence $(X_n)_{n\le0}$ of split words is indexed by the nonpositive integers. For every $n \le 0$, the law of $X_n$ is uniform on the set of words of length $\ell_n$ on $A$. Moreover, if one splits the word $X_{n-1}$, whose length is $\ell_{n-1}$, into $r_n$ subwords of length $\ell_{n}$, then the word $X_n$ is chosen uniformly among these subwords, independently of everything up to time $n-1$. More precisely, denote by $V_n$ the location of the subword $X_{n}$ in $X_{n-1}$. Then $V_n$ is uniform in $\{1,2,\ldots,r_n\}$. The split-word process is $(X_n,V_n)_{n\le0}$.

Call $\F^{X,V}=(\F^{X,V}_n)_{n\le 0}$ the natural filtration of $(X_n,V_n)_{n \le 0}$. Clearly, every subsequence $(\F^{X,V}_n)_{n\in Q}$ with $Q \subset \z^-$ is the natural filtration of a split-word process with lengths $(\ell_n/\ell_m)_{n \in Q}$ on the alphabet $A^{\ell_m}$, where $m=\max Q$.

The filtration $\F^{X,V}$ is $(r_n)$-adic since for every $n \le 0$, $$\F^{X,V}_n=\F^{X,V}_{n-1}\vee \sigma(V_n)\text{ with } V_n \text{ independent of } \F^{X,V}_{n-1}.$$
Moreover, the tail $\sigma$-field $\F^{X,V}_{-\infty}$ is trivial, thanks to proposition 6.2.1 in~\cite{laurent2004ftd}. Yet, the inclusion $\F^V_n \subset \F^{X,V}_n$ is clearly
strict since $X_0$ is independent of $(V_n)_{n\le 0}$.
However, the filtration $\F^{X,V}$ may still be a product type filtration (generated by some other independent sequence).

\subsection{Results}

Surprisingly, the nature of the filtration $\F^{X,V}$ depends on the sequence $(\ell_n)_{n \le0}$:

\begin{theoreme}\label{theoA}
The filtration $\F^{X,V}$ is standard (or equivalently, is of product type) if and only if, the series
$\displaystyle\sum_{n}\frac{ \ln(r_{n})}{ \ell_{n}}$
diverges.
\end{theoreme}

The `if' part of the theorem, which is new, will be proved in section~\ref{S2-split}, and the `only if' part in section~\ref{S3-split}. \\
Note that the convergence of the series $\displaystyle\sum_{n}\frac{ \ln(r_{n})}{ \ell_{n}}$ is equivalent to the condition $\Delta$ of Laurent~\cite{laurent2004ftd}, who established the `only if' part of theorem~\ref{theoA}.

The condition that we call $\neg\Delta$ (the divergence of the series $\displaystyle\sum_{n}\frac{ \ln(r_{n})}{ \ell_{n}}$), which implies standardness, improves on Laurent's sufficient condition of standardness:
\begin{itemize}
\item[($\nabla_N$)]
There exists $\alpha<1$ such that $r_{n}^{\alpha}\gg N^{\ell_{n}}$ as $n \to -
\infty$.
\end{itemize}


Laurent notices that conditions $\nabla_N$ are weaker than the condition
$$\frac{ \ln(r_n)}{ \ell_{n}} \to \infty\text{ as }n \to -\infty, \quad \quad \quad (\nabla)$$
which does not depend on $N$.
Laurent indicates that conditions $\nabla$ and $\Delta$ were previously introduced by Vershik~\cite[example~1]{vershik1995tds}, in the context of decreasing sequences of measurable partitions.

Since conditions $\nabla$ and $\Delta$ do not exhaust all possible situations, both Vershik and Laurent asked what happens ``between $\nabla$ and $\Delta$''. In theorem~\ref{theoA}, we solve Laurent's question: condition $\Delta$ is in fact necessary and sufficient for the filtration of the split-word process to be non standard.

After the completion of this paper, Anatoly Vershik drew our attention to the paper \cite{heicklen}, where Heicklen obtains a result equivalent to Theorem~\ref{theoA}. The $(r_n)$-adic filtrations studied by Heicklen were introduced by Vershik \cite{vershik1970} as follows. (We mention that both Vershik and Heicklen index sequences by the set of nonnegative integers whereas we index them by the set of nonpositive integers, and that this is the only difference between their presentation and ours, given below.)

Let $A$ be a finite alphabet and $(G_n)_{n \le 0}$ the decreasing sequence of groups defined by
$$
G_n = \sum_{k=n+1}^0 \z/r_k\z.
$$
Let $G$ denote the union over $n \le 0$ of the groups $G_n$. For every $n \le 0$, the group $G_n$ acts on $A^G$ (on the left) by canonical shifts. Namely, for every $g \in G_n$ and $f \in A^G$, one defines $g \cdot f\in A^G$ by $(g \cdot f) (x) = f(xg)$  for every $x \in G$. Let ${\rm Orb}_{G_n}(f)$ denote the orbit of a given $f\in A^G$ under the action of $G_n$, that is, ${\rm Orb}_{G_n}(f)=\{g \cdot f\,;\,g\in G_n\}$.

Let $F  = (F(x))_{x \in G} \in  A^G$ be a random function whose coordinates $F(x)$ are independent and uniformly distributed in $A$. The filtration $({\mathcal F}_n)_{n \le 0}$ studied by Vershik and Heicklen is the natural filtration of $(\mathrm{O}_n)_{n \le 0}$, where
$$
\mathrm{O}_n={\rm Orb}_{G_n}(F).
$$
For each $n \le 0$, $\mathrm{O}_{n-1}$ is the union of $r_n$ orbits under the action of $G_n$, one orbit for each element of $G_{n-1}/G_n$. Almost surely, these orbits are all different since the shifted functions $g \cdot F$ are different. Futhermore, conditionally on $(\mathrm{O}_{n-1},\mathrm{O}_{n-2},...)$, the random variable $\mathrm{O}_n$ is uniformly distributed on these $r_n$ orbits. This shows that $({\mathcal F}_n)_{n \le 0}$ is an $(r_n)$-adic filtration.

One can show that the tail $\sigma$-field ${\mathcal F}_{-\infty}$ is trivial and that the filtration of the split-words process is immersed in $({\mathcal F}_n)_{n \le 0}$. Informally, the word $X_n$ at time $n$ is given by the values at $e$ (the identity of the group $G$) of the elements of $\mathrm{O}_n$. One gets $X_n$ from $X_{n-1}$ by splitting the orbits under $G_{n-1}$ into $r_n$ orbits under $G_n$ and by choosing one of these orbits uniformly randomly.

Thus, the standardness of $({\mathcal F}_n)_{n \le 0}$ when condition $\Delta$ fails implies that the natural filtration of  the split-words process is standard and, therefore, that it is of product type. Heicklen's proof relies on Vershik's standardness criterion and uses the language of ergodic theory. Although Heicklen's result and our Theorem~\ref{theoA} are logically equivalent, we believe that our proof is interesting because it relies on a constructive, direct and probabilistic method.


This result has interesting applications in ergodic theory as we now explain. Recall that entropy is a well known invariant associated to an automorphism of a probability space (that is, a bimeasurable application preserving the measure).
Vershik~\cite{vershik1973four} defined a much more elaborate invariant, named the scale. But computing this invariant is a very difficult task, even in simple cases. However, Laurent showed that our result provides the exact scale of a dyadic transformation~\cite{laurent2009vershikian}: more precisely, the scale of this transformation is the set of sequences $(r_n)_{n \le 0}$ fulfilling condition $\Delta$.

Using Vershik's theory, one can deduce from theorem~\ref{theoA}
that the filtration of the split-word process on any separable alphabet (endowed with an arbitrary measure) is standard under condition $\neg\Delta$, see \cite{laurent2009pre} for a proof.


\subsection{Examples}
\label{S2.1-split}

Condition $\neg\Delta$ forces the length $\ell_n$ to grow very quickly as $n$ goes to $-\infty$. If a given sequence $(\ell_n)_{n\le0}$ is $\neg\Delta$, then every sequence $(\ell'_n)_{n\le0}$ such that $(\ln r'_n)/\ell'_n \ge (\ln r_n)/\ell_n$ is $\neg\Delta$ as well. No similar property holds for the sequence $(\ell_n)_{n\le0}$ only, nor for the sequence $(r_n)_{n\le0}$ only. Indeed, the first example of this section provides a sequence $(r_n)_{n}$ which is $\neg\Delta$ and such that $(r^2_n)_{n\le0}$ is $\Delta$.

\begin{exemple}[A standard split-word filtration]
Let $\ell_{0}=1$ and $\ell_{n-1}=2^{\ell_{n}}$ for every $n \le 0$. That is to say $$\ell_{n}=2^{2^{2^{.^{.^{.^{2}}}}}} \text{ where the figure } 2 \text{ appears } |n|  \text{ times.}$$
Then for every $n<0$, $$r_{n}=\ell_{n-1}/\ell_{n}=2^{\ell_{n}-\ell_{n+1}}.$$
Therefore $$\log_2(r_{n})/\ell_{n} = \frac{\ell_{n}-\ell_{n+1}}{\ell_{n}} \to 1,$$
which proves that $(r_n)$ is $\neg\Delta$.
\end{exemple}

Theorem~\ref{theoA} has another interesting consequence which we now explain. Recall that, by the lacunary isomorphism theorem~\cite{emery2001vershik}, from any filtration $(\F_n)_{n \le 0}$ (such that $\F_0$ is essentially separable), one can extract a filtration $(\F_n)_{n \in Q}$ which is standard. In \cite{vershik1995tds}, Vershik provides an example where $(\F_n)_{n \le 0}$ is non standard whereas $(\F_{2n})_{n \le 0}$ is standard. In~\cite{gorbulsky1999one}, Gorbulsky also gives such an example. Theorem~\ref{theoA} provides an example of a non standard filtration (example~\ref{example_non-standard} below) in which the transition from the non standard case to the standard case is very sharp: $(\F_n)_{n \in Q}$ is standard for every infinite subset $Q$ of $\z^-$ with infinite complementary. 

\begin{exemple}[A non standard filtration close to standardness]\label{example_non-standard}
Set $\ell_{0}=1$ and $\ell_{n-1}=4^{\sqrt{\ell_{n}}}$ for every $n \le 0$. That is to say
$$\ell_n=4^{2^{2^{.^{.^{.^{2}}}}}} \text{ where the figure  } 2 \text{ appears } |n|-1 \text{ times.}$$
Then the filtration $(\F_n^{(X,V)})_{n \le 0}$ is not of product type.
Yet, if $\phi$ is a strictly increasing application from $\Zm$ into $\Zm$ such that $\phi(n)-n \to -\infty$ as $n \to - \infty$, then the filtration   $(\F^{(X,V)}_{\phi(n)})_{n\le 0}$ is of product type.
\end{exemple}

\begin{proof}
On the one hand
$$\frac{\log_2 r_n}{\ell_n}\le\frac{\log_2\ell_{n-1}}{\ell_n}=\frac{2}{\sqrt{\ell_n}},
$$
which is the general term of a convergent series, thus $(\ell_n)_{n\le 0}$ is $\Delta$.
On the other hand the filtration $(\F^{(X,V)}_{\phi(n)})_{n\le 0}$ is the filtration of a split-word process of length process $(\ell'_n)_{n \le 0}=(\ell_{\phi(n)})_{n \le 0}$. The ratios between successive lengths are, for $n\le0$,
$$
r'_n=\ell'_{n-1}/\ell'_n=\ell_{\phi(n-1)}/\ell_{\phi(n)}.
$$
If $\phi(n)-n \to -\infty$ when $n \to - \infty$, then $\phi(n-1) \le \phi(n)-2$ infinitely often. For these $n$,
$$r'_n \ge \frac{\ell_{\phi(n)-2}}{\ell_{\phi(n)}}=\frac{4^{\sqrt{\ell_{\phi(n)-1}}}}{\ell_{\phi(n)}}=\frac{4^{2^{\sqrt{\ell_{\phi(n)}}}}}{\ell_{\phi(n)}}$$
hence
$$ \frac{\log_2 r'_n}{\ell'_n} \ge 2\frac{2^{\sqrt{\ell_{\phi(n)}}}}{\ell_{\phi(n)}}-\frac{\log_2 \ell_{\phi(n)}}{\ell_{\phi(n)}}.
$$
This shows that a subsequence of $(\log_2(r'_n)/\ell'_n)_n$ converges to infinity,
hence that $(\ell'_n)_n$ is $\neg\Delta$.
\hfill $\square$
\end{proof}


\section{Laurent's method and tools}
\label{S1-split}
In this section, we introduce the tools used by Laurent to prove that under condition $\nabla$, the filtration of the split-word process is of product type. Laurent used a canonical coupling to build explicitly a sequence of innovations $(V'_n)_{n \le0}$ which generates the process $(X_{n},V_n)_{n \le 0}$.

\begin{definition}
If $(\F_n)_{n \le 0}$ is a filtration, and $(U_n)_{n\le0}$ is a sequence of random variables such that for every $n \le 0$,
$$\F_n=\F_{n-1}\vee \sigma(U_n)\ \ \text{ with } U_n \text{ independent of } \F_{n-1},$$
one says that $(U_n)_{n \le 0}$ is a sequence of innovations for $(\F_n)_{n\le0}$.
\end{definition}

This method is strengthened in section~\ref{S2-split}, where we consider a partial canonical coupling to improve on condition $\nabla$.

We remind the reader that sequences are indexed by the nonpositive integers.

\subsection{Change of innovations}
\label{S1.1-split}
We start with a complete definition of split-word processes.

\begin{definition}[Split-word process]
Let  $(r_{n})_{n \le 0}$ denote a sequence of integers such that $r_n\ge 2$ for every $n\le0$. Set
 $\ell_{0}=1$ and, for every $n \le 0$, $\ell_{n-1}=r_n \ell_n$.
Let $A$ denote a finite set, called the alphabet, with cardinal $N \ge
2$.

A split-word process is any process $(X_{n},V_{n})_{n \le 0}$  such that, for every $n \le 0$,
\begin{itemize}
\item $X_{n}$ is uniformly distributed on $A^{\ell_n}$,
\item $V_{n}$ is uniformly distributed on $\{1,\ldots,r_{n}\}$ and independent of the $\sigma$-algebra ${\cal F}^{(X,V)}_{n-1} =
  \sigma(X_{m},V_{m}\,;\,m \le n-1)$,
\item if the word $X_{n-1}$ (with length $l_{n-1}=l_n r_n$) is partitioned into $r_n$ subwords of length $l_n$, $X_n$ is the $V_n$-th among those $r_n$ subwords.
\end{itemize}
\end{definition}

The sequence $(X_{n})_{n \le 0}$ is a inhomogeneous Markov process indexed by the negative integers and generated by the
innovations $(V_{n})_{n \le 0}$. The existence of such a process $(X_{n},V_n)_{n \le 0}$ is guaranteed by Kolmogorov's theorem.

To prove that, under some conditions, the filtration of the split-word process is of product type, one has to switch from one set of innovations to another. Lemma~\ref{lem11} provides a general method to build new innovations.

\begin{lemme}[Change of innovations]
\label{lem11}
For every $n \le -1$, let $\{\varphi^{n}_{w}\}_{w}$ denote a family of permutations of $\{1,\ldots,r_{n+1}\}$, indexed by the elements $w$ of $A^{\ell_{n}}$, and let
$$
V'_{n+1}=\varphi^{n}_{X_{n}}(V_{n+1}).
$$
Then $(V'_{n})_{n \le 0}$ is a sequence of generating innovations for $(X_{n})_{n \le 0}$. This means that, for every negative integer $n \le -1$, the following properties hold:
\begin{itemize}
\item The random variable $V'_{n+1}$ is uniformly distributed on $\{1,\ldots,r_{n+1} \}$.
\item The random variable $V'_{n+1}$ is independent of $\F_{n}^{X,V}$ and therefore also of $\F_{n}^{X,V'}$.
\item The random variable $X_{n+1}$ is a measurable function of $X_{n}$ and $V'_{n+1}$.
\end{itemize}
\end{lemme}

\begin{proof}[of lemma~\ref{lem11}]
For every negative integer $n$ and every $v$ such that $1\le v\le r_{n+1}$,
a simple computation proves that
$$
\P[V'_{n+1}=v|\F_{n}^{X,V}]=\P[V_{n+1}=(\varphi^n_{X_n})^{-1}(v)|\F_{n}^{X,V}]
=
1/r_{n+1}.
$$
This shows the first two properties. The third property follows from the fact that $X_{n+1}$ is the $k$-th subword of $X_n$, where $k=(\varphi^n_{X_n})^{-1}(V'_n)$.
\hfill $\square$ \end{proof}

\subsection{Canonical word and coupling}\label{S1.2-split}
To build the innovations which generate the process $(X_{n})_{n\le 0}$, one can use, and improve on, Laurent's construction under the stronger condition $\nabla$. This uses the notions of canonical word and canonical coupling, which we recall below.

\begin{definition}[Canonical alphabets and canonical words]
For every integer $M\ge2$, the canonical alphabet on $M$ letters is $A_M=\{1,\ldots,M\}$.
Canonical words on $A_M$ are the words whose $i$-th letter is congruent to $i$ modulo $M$.
Hence, the letters of $A_M$ appear in order and are repeated periodically.
\end{definition}
Canonical words will usually be denoted by the letter $c$.
For example the canonical word of length $11$ on $A_3$ is $12312312312$.

\begin{notation}[General notations]
To simplify the definition of the canonical coupling, one identifies any ordered alphabet $B$ of size $M\ge2$ with $A_M$ according to the rank of each letter in the alphabet $B$.\\
The $i$-th letter of a word $w$ is denoted by $w(i)$.
Let $w=(w(i))_{1\le i\le r}$ denote a word of length $r$. For every $1\le i\le r$, $H(w,i)$ denotes the number of instances of the letter $w(i)$ among the $(i-1)$ first letters of $w$:
$$H(w,i)=\sum_{1 \le j < i} \mathbf{1}_{\{w(j)=w(i)\}}.$$
\end{notation}

\begin{definition}[Canonical coupling]
Let $w$ denote a word of length $r$ on an ordered alphabet $B$ of size $M\ge2$.
The canonical coupling associated to $w$ is the permutation $\varphi_{w}$ of $\{1,\cdots,r\}$ defined as follows: for every $i \le r$,
$$
\varphi_{w}(i) = w(i) + H(w,i)M
\quad\mbox{if}\quad
w(i) + H(w,i)M\le r.
$$
After this process has been applied to every $i$, one chooses $\varphi_{w}(j)$ for the integers $j$ such that $w(j) + H(w,j)M>r$, in an increasing way and in order to make $\varphi_{w}$ a bijection. (So  $\varphi_{w}(j)$ is the smallest $k$ which does not belong yet to the range of   $\varphi_{w}$.)
\end{definition}

Later on, we apply the notions of canonical word and canonical coupling to some alphabets $A^\ell$ with $\ell\ge1$.

\begin{figure}[h]
\begin{center}
\includegraphics[width=.75\textwidth]{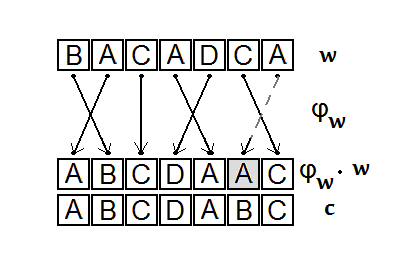}
\caption{Example of canonical coupling}
\end{center}
\label{exemple de couplage}
\end{figure}


By construction $\varphi_{w}$ is one of the permutations $\varphi$ such that $\varphi \cdot w := w \circ \varphi^{-1}$ is as close as possible to a canonical word. Lemma~\ref{lem12} makes this statement more precise.

\begin{lemme}[Comparison of $w$ and $c \circ \varphi_{w}$]
\label{lem12}
Let $c$ be the canonical word of length $r \ge 1$ on an ordered alphabet $B$ of size $M\ge2$. Then for every $1\le i\le r$ and $w \in B^r$,
$$
c(\varphi_{w}(i))=w(i)
\qquad\mbox{if and only if}\qquad
w(i) + H(w,i)M \le r. $$
\end{lemme}

\begin{proof}[of lemma~\ref{lem12}]
By definition of $c$, the number of instances of $j$ in $c$ is
$$
K_j=\max\{k\ge0 : j+(k-1)M\le r\}=\lfloor \frac{r-j}{M} \rfloor+1 =\lceil\frac{r+1-j}{M}\rceil.
$$
If $w(i) + H(w,i)M \le r$, then $\varphi_{w}(i)=w(i) + H(w,i)M$, thus $c(\varphi_{w}(i))=w(i) $ by definition of $c$.

Otherwise, $w(i) + H(w,i)M > r$, hence the number $H(w,i)$ of instances of $w(i)$ among the $i-1$ first letters of $w$ is at least the number $K_{w(i)}$ of instances of the letter $w(i)$ in $c$. According to the first case, $\varphi_w$ sends the ranks of the $K_{w(i)}$ first instances of $w(i)$ in $w$ on the ranks of the instances of $w(i)$ in $c$. Since $\varphi_w$ is bijective, the rank of every instance of the letter  $w(i)$ in $c$ has an antecedent by $\varphi_{w}$ that is less than $i$. Therefore the letter in $c$ with rank $\varphi_{w}(i)$ cannot be $w(i)$, that is, $c(\varphi_{w}(i)) \ne w(i).$
\hfill $\square$ \end{proof}

Lemma~\ref{lem12} implies lemma~\ref{lem13} below, which is a slight improvement on Laurent's lemma~6.3.2 of~\cite{laurent2004ftd}.
\begin{lemme}
\label{lemme technique 1}\label{lem13}
Let $X$ denote a uniform random word of length $r$ on an ordered alphabet $B$ of size $M\ge2$, $V$ a uniform random variable on  $\{1,\ldots,r\}$,  independent of $X$, and $c$ the canonical word of length $r$ on $B$. Then,
$$
\P[X(V) \ne c(\varphi_{X}(V))] \le M/r+2(M/r)^{1/3}.
$$
\end{lemme}

\begin{proof}[of lemma~\ref{lem13}]
Since $X(V)\le M$ and
$\{X(V)=c(\varphi_{X}(V))\}=\{X(V)+H(X,V)M \le r\}$,
$$
\{X(V)=c(\varphi_{X}(V))\}
\supset
\{H(X,V)M \le r-M\},
$$
hence, for every positive $s$,
$$
\{X(V)=c(\varphi_{X}(V))\}
\supset
\{H(X,V)M\le V-1+Ms\}\cap\{V-1+Ms\le r-M\}.
$$
Taking complements, one gets
$$
\{X(V)\ne c(\varphi_{X}(V))\}
\subset
\{H(X,V)M>V-1+Ms\}\cup\{V-1+Ms>r-M\},
$$
which yields
$$
\P[X(V)\ne c(\varphi_{X}(V))]
\le
\P[H(X,V)M>V-1+Ms]+\P[V-1+Ms>r-M].
$$
Since $r+1-V$ and $V$ are both uniform on   $\{1,\ldots,r\}$,
$$
\P[V-1+Ms>r-M]=\P[V<M(s+1)]\le M(s+1)/r.
$$
On the other hand,
$$
\P[H(X,V)M>V-1+Ms]\le \P\left[ \big| H(X,V)-(V-1)/M \big| >s\right].
$$
Since $V$ and $X$ are independent,
conditionally on $V$
 the distribution of $H(X,V)$ is binomial Bin$(V-1,1/M)$.
The conditional expectation of $H(X,V)$ is $(V-1)/M$ and the conditional variance of $H(X,V)$ is
$$
(V-1)(M-1)/M^2\le V/M,
$$
hence the
Bienaym\'e-Chebychev inequality yields
$$
\P[H(X,V)M>V-1+Ms]\le
\mathbb{E}[V]/(Ms^2)=(r+1)/(2Ms^2)\le r/(Ms^2).
$$
Finally,
$$
\P[X(V)\ne c(\varphi_{X}(V))]\le M(s+1)/r+r/(Ms^2).
$$
This upper bound for $s=(r/M)^{2/3}$ implies the statement of the lemma.
\hfill $\square$\end{proof}

\subsection{Adapting Laurent's proof}\label{S1.3-split}

Let us state a slight improvement on Laurent's result~\cite[proposition 6.3.3]{laurent2004ftd}:
\begin{theoreme}
\label{theoB}
If $r_{n}\gg N^{\ell_{n}}$ as $n \to -
\infty$, then the natural filtration of $(X_{n},V_n)_{n \le 0}$ is of product type.
\end{theoreme}

Let us introduce notations for the proof. From now on, a word of length $\ell_{n}$ on the alphabet $A$ will often be seen as a word of length $r_{n+1}$ on the alphabet $A^{\ell_{n+1}}$.

{\parindent0cm\textit{Proof of theorem \ref{theoB}.}}
\begin{itemize}
\item Choice among sub-words: for every $x$ in $A^{\ell_{n-1}}$ and $v$ in $\{1,\ldots,r_n\}$, denote by $f_{n}(x,v)=x(v)$ the $v$-th letter of $x$ seen as a word of length $r_n$ on the alphabet $A^{\ell_{n}}$. So $X_n=f_n(X_{n-1},V_n)$.
\item New innovations: let $c_{n-1}$ be the canonical word of length $r_{n}$ on the alphabet $A^{\ell_{n}}$ (for some fixed order on $A^{\ell_{n}}$). Let $\varphi_{X_{n-1}}$ be the canonical coupling associated to $X_{n-1}$ seen as a word of length $r_{n}$ on $A^{\ell_{n}}$. Set $V'_{n}=\varphi_{X_{n-1}}(V_{n})$.
\item Construction of a sequence $(X'_n)_{n \le 0}$ approximating $(X_n)_{n \le 0}$: set
$$X'_{n}=f_{n}(c_{n-1},V'_{n})=f_{n}(c_{n-1},\varphi_{X_{n-1}}(V_{n})).$$
\end{itemize}

The key point is to show that
$$ \P [X_{n} = {X}'_{n}] \to 1 ~\mbox{ as }~ n \to -\infty$$
by bounding above
$$ \P [X_{n-1}(V_{n}) \ne c_{n-1}(\varphi_{X_{n-1}}(V_{n})) ]. $$

Applying lemma~\ref{lemme technique 1} to $X_{n-1}$ seen as a word of length $r_{n}$ on the alphabet $A^{\ell_{n}}$, one gets that,
$$ \P[X_{n} \ne X_{n}'] \le N^{\ell_{n}}/r_{n}+2(N^{\ell_{n}}/r_{n})^{1/3}.$$

Hence $\P[X_{n} \ne X'_{n}]\to0$ since $N^{\ell_{n}} \ll r_{n}$.

Define an application $f'_{n} : A^{\ell_{n-1}} \times \{1,\ldots,r_n\} \to A^{\ell_n}$ by
$f'_{n}(x,v')=f_{n}(x,\varphi^{-1}_x(v'))$. Then $f'_{n}(X_{n-1},V'_n)=X_n$ and $f'_{n}(~\cdot ~,V'_{n})\circ\ldots\circ f'_{m+1}(~\cdot ~,V'_{m+1})(X_{m})=X_n$ for $m \le n \le 0$.
Therefore, under the assumption $\nabla$, one has,
\begin{eqnarray*}
\P\left[X_{n} \ne f'_{n}(~\cdot ~,V'_{n})\circ\ldots\circ f'_{m+1}(~\cdot ~,V'_{m+1})(X'_{m})\right] &\le& \P[X_{m} \ne X'_{m} ] \\
&\to& 0 \text{ as } m \to -\infty.
\end{eqnarray*}
This implies the convergence in probability:
\begin{eqnarray*}
X_{n}&=&\lim_{ m \to -\infty} f'_{n}(~\cdot ~,V'_{n})\circ f'_{n-1}(~\cdot ~,V'_{n-1})\circ \ldots\circ f'_{m+1}(~\cdot ~,V'_{m+1})(X'_{m}) \\
&=&\lim_{ m \to -\infty} f'_{n}(~\cdot ~,V'_{n})\circ \ldots\circ f'_{m+1}(~\cdot ~,V'_{m+1})\circ f_{m}(~\cdot ~,V'_{m})(c_{m-1})
\end{eqnarray*}
and proves that the innovations $(V'_{k})_{k \le n}$ determine the words $(X_{k})_{k \le n}$ and the innovations $(V_{k}=\varphi_{X_{k-1}}^{-1}(V'_{k}))_{k \le n}$. Thus, the filtration $(\F^{(X,V)}_{n})_{n\le 0}$ is generated by the innovations $(V'_{n})_{n \le 0}$.
\hfill $\square$

\section{Improving on condition $\nabla$}
\label{S2-split}
\subsection{Statement of the main result}\label{S2.2-split}

As said before, there is a gap between the conditions $\nabla$ and $\Delta$ under which the problem has been solved by Laurent. Our next theorem bridges the gap between the two conditions.

\begin{theoreme}
\label{theoC}
If the series $\displaystyle\sum_{n}\ln(r_{n})/\ell_{n}$ diverges (condition $\neg\Delta$),
then the filtration $(\F^{(X,V)}_{n})_{n\le 0}$ is of product type.
\end{theoreme}

\begin{remarque} Laurent~\cite{laurent2004ftd} states condition $\Delta$ as the convergence of the series
$$
\sum_{n} \frac{\ln(r_{n}!)}{\ell_{n-1}}.
$$
The inequalities $ \frac12 r\ln(r) \le \ln(r!) \le r \ln(r)$, valid  for every $r\ge 2$, ensure that the series
$$
\sum_{n}\frac{ \ln(r_{n})}{ \ell_{n}}
\quad\mbox{and}\quad
\sum_{n} \frac{\ln(r_{n}!)}{\ell_{n-1}}
$$
both converge or both diverge. Hence Laurent's condition $\Delta$ and the condition $\Delta$ which we stated in our introduction and used since, are indeed equivalent.
\end{remarque}

Condition $\neg\Delta$ is easy to express, but less handy to prove things. The equivalent wording below, which is closer to theorem~\ref{theoB}, is more convenient.

\begin{proposition}[Rewording of condition $\neg\Delta$]
\label{p21}
Condition $\Delta$ fails if and only if there exists a sequence $(\alpha_{n})_{n \le 0}$ of nonnegative real numbers and an increasing application $\phi:\Zm\to\Zm$, such that the following properties hold:
\begin{enumerate}
\item
For every $n\le0$,
$r_{\phi(n)}\ge N^{2\alpha_{n} \ell_{\phi(n)}}.$
\item
The series
$\displaystyle\sum_n \alpha_{n}$ diverges.
\end{enumerate}
Furthermore, when they exist, the sequence $(\alpha_{n})_{n \le 0}$ and the application $\phi$ can be chosen in such a way that the additional properties below hold:
\begin{enumerate}
\item[3.]
When $n\to-\infty$, $r_{\phi(n)}\gg N^{2\alpha_{n} \ell_{\phi(n)}}$ (in particular, $r_{\phi(n)} \to +\infty$).
\item[4.]
The series
$\displaystyle\sum_n \alpha_{2n}$ diverges.
\item[5.]
For every $n\le0$, $0<\alpha_n\le1$.
\item[6.]
For every $n\le-2$, the ratio $\alpha_n \ell_{\phi(n)}/\ell_{\phi(n+1)}$ is an integer.
\end{enumerate}
\end{proposition}

The proof of this result can be found in section~\ref{aux-split}.

\subsection{Construction of the new innovations}
\label{S2.3-split}

The construction of the new innovations uses a \textit{partial canonical coupling}. This tool sharpens the canonical coupling that was introduced in section~\ref{S1.2-split} and which has to be kept in mind.

Under the hypotheses of theorem \ref{theoC}, the ratios $(r_{n})_{n \le 0}$ are no longer big enough for the innovations associated to the canonical coupling to approach the entire word $X_n$ in only one step. Therefore several steps are necessary to get a good information on the word.

\begin{definition}[Partial canonical coupling]\label{couplage canonique partiel}
Let $w$ be a word of length $\ell r$ on the alphabet $A$ and $\lambda \in \{1,\cdots,\ell\}$ an integer. Denote by $\tilde{w}$ the word extracted from $w$ by splitting $w$ into $r$ sub-words of length $\ell$ and keeping only the first $\lambda$ letters of each sub-word.
In other words, if $w=(w_{1},\ldots,w_{\ell r})$, then
$$\tilde{w} = (w_{1},\ldots,w_{\lambda},w_{\ell+1},\ldots,w_{\ell+\lambda},\ldots,w_{(r-1)\ell+1},\ldots,w_{(r-1)\ell+\lambda})$$
is the word constituted of the letters $w_i$ such that $i=j$ mod $\ell$ with $1\le j\le\lambda$.

Let $\varphi_{\tilde{w}}$ be the canonical coupling of $\tilde{w}$ towards the canonical word $\tilde{c}$ of length $r$ on the alphabet $A^{\lambda}$.
Since $\varphi_{\tilde{w}}$ belongs to $\mathfrak{S}_{r}$ (the symmetric group on r letters), one can apply it to the $r$ sub-words of $w$ of length $\ell$.
The permutation $\varphi_{w}^{\lambda/\ell}=\varphi_{\tilde{w}}$ is called the partial canonical coupling of rate $\lambda/\ell$ associated to $w$.
\end{definition}

\vfill
\eject

\begin{figure}[h]
\begin{center}
\label{fig2}
\includegraphics[width=.95\textwidth]{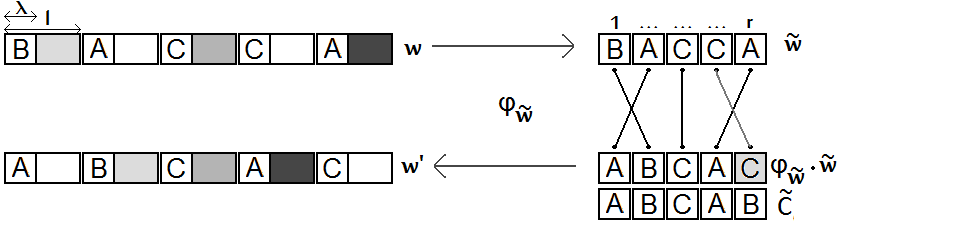}
\caption{Example of a partial canonical coupling. In the proof, this coupling will be considered with $\ell=\ell_{\phi(2n)}$, $\lambda=\alpha_{2n}$,~~$r=r_{\phi(2n)}$,  $w=X_{\phi(2n)-1}$ and $\tilde{w}=\tilde{X}_{\phi(2n)-1}$.}
\end{center}
\label{Partial canonical coupling}
\end{figure}

{\parindent0cm\textit{Proof of theorem \ref{theoC}.}}
Assume that $\neg\Delta$ holds. Fix $\phi$ and $(\alpha_n)_{n \le 0}$ which fulfill conditions 1,2,3,4,5 and 6 of proposition~\ref{p21}. We now construct new innovations $(V'_n)_{n \le 0}$ which generate the same filtration than the process $X$.

Let us define new innovations $(V'_{k})_{k \le 0}$ as follows. For every $k\le0$, define
$V'_k= \varphi_{X_{k-1}}^{\alpha_{2n}}(V_{k})=\varphi_{\tilde{X}_{k-1}}(V_{k})$ if there exists an integer $n$ (necessarily unique) such that $k=\phi(2n)$, and
$V'_k=V_k$ otherwise.

Lemma 2.3 will be used to show that with probability close to $1$, the first $\alpha_{2n} \ell_{\phi(2n)}$ letters of the words $X_{\phi(2n)}$ and those of  $f_{\phi(2n)}(C_{\phi(2n)-1},V'_{\phi(2n)})$ coincide.

\vfill
\eject

\begin{figure}[h]
\begin{center}
\label{fig3}
\includegraphics[width=.95\textwidth]{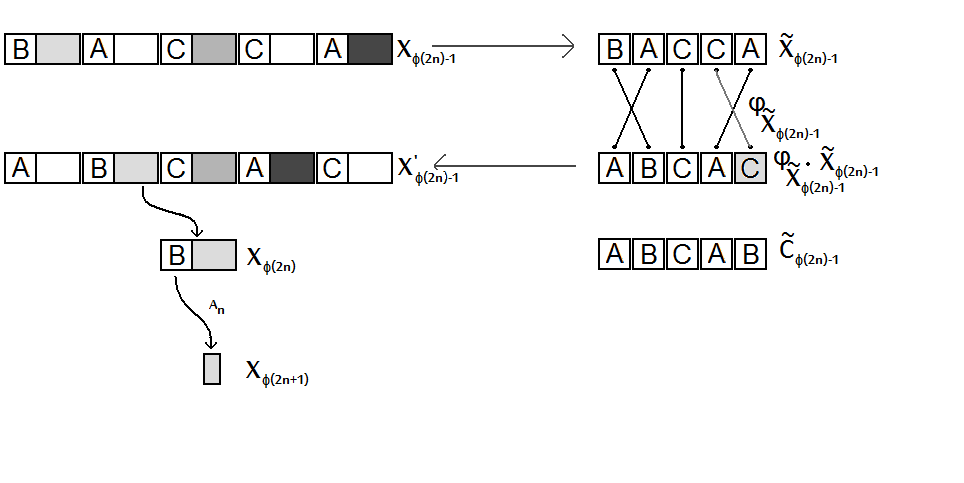}
\caption{Description of the method.}
\end{center}
\label{Description methode}
\end{figure}

\subsection{Proof of the main result}\label{S2.4-split}

This proof is split into three steps.

\paragraph{First step: $X_{\phi(2n+1)}$ comes from the beginning of $X_{\phi(2n)}$ infinitely often}\ \\

One focuses on the events
$$A_{n}=\{ X_{\phi(2n+1)} \text{ comes from the first } \alpha_{2n}\ell_{\phi(2n)} \text{ letters of }X_{\phi(2n)} \}. $$

One computes $\P[A_{n}]$ by counting: the number of possible choices for the innovations  $V_k$ for $\phi(2n)+1\le k\le\phi(2n+1)$ is
$$
\prod_{k=\phi(2n)+1}^{\phi(2n+1)}r_k
=
\ell_{\phi(2n)}/\ell_{\phi(2n+1)}.
$$
The number of cases such that $A_{n}$ occurs is the number of sub-words of length $\ell_{\phi(2n+1)}$ entirely included in the first $\alpha_{2n}\ell_{\phi(2n)}$ letters of $X_{\phi(2n)}$: this number is $\alpha_{2n}\ell_{\phi(2n)}/\ell_{\phi(2n+1)}$ thanks to the additional hypothesis that $\alpha_{2n}\ell_{\phi(2n)}/\ell_{\phi(2n+1)}$ is an integer. Therefore $\P [A_n]=\alpha_{2n}$
and the series $ \sum_{n} \P[A_{n}]$ diverges.

Moreover $A_{n}$ is a (deterministic) function of $V_k$ for $\phi(2n)+1\le k\le\phi(2n+1)$, hence the events $A_{n}$ are independent and the Borel-Cantelli lemma ensures that almost surely, $A_{n}$ occurs for infinitely many $n$.

Note that $A_n \in \F_{\phi(2n+1)}^{V'}$ for every $n$ since $V'_k=V_k$ for every time $k$ which is not one of the integers $\phi (2n)$.

\paragraph{Second step: Use of lemma~\ref{lemme technique 1}} \ \\

Our purpose is to prove lemma~\ref{lem24}.
\begin{lemme}\label{lem24}
For every $n \le 0$, set $I_{\phi(2n)}=\{1,...,\alpha_{2n}\ell_{\phi(2n)}\}$ and fix a word $C_{\phi(2n)-1}$ of length $\ell_{\phi(2n)-1}=r_{\phi(2n)}\ell_{\phi(2n)}$ on $A$ such that $\tilde{C}_{\phi(2n)-1}$ is the canonical word of length $r_{\phi(2n)}$ on the alphabet $A^{\alpha_{2n} \ell_{\phi(2n)}}$. The probability for $X_{\phi(2n)}(I_{\phi(2n)})$ (the first $\alpha_{2n}\ell_{\phi(2n)}$ letters of $X_{\phi(2n)}$) to be the $V'_{\phi(2n)}$-th sub-word of $\tilde{C}_{\phi(2n)-1}$ converges to $1$ as $n$ tends towards $-\infty$.
That is to say,
$$ \P\left[X_{\phi(2n)}(I_{\phi(2n)}) = f_{\phi(2n)}(C_{\phi(2n)-1},V'_{\phi(2n)})(I_{\phi(2n)})\right] \to 1, \text{ as } n \to - \infty.$$
\end{lemme}

\begin{proof}[of lemma~\ref{lem24}]
Note that $X_{\phi(2n)}(J_{\phi(2n)})$ is the $V_{\phi(2n)}$-th letter of $\tilde{X}_{\phi(2n)-1}$ seen as a word of length $r_{\phi(2n)}$ on the alphabet $A^{\alpha_{2n} \ell_{\phi(2n)}}$ (where $\tilde{X}_{\phi(2n)-1}$ is built from $X_{\phi(2n)-1}$ according to definition~\ref{couplage canonique partiel} and the caption of figure~\ref{fig2}). Hence
\begin{eqnarray*}
& & \P\left[X_{\phi(2n)}(I_{\phi(2n)}) = f_{\phi(2n)}(C_{\phi(2n)-1},V'_{\phi(2n)})(I_{\phi(2n)})\right]\\
& &\qquad\qquad\qquad= \P\left[\tilde{X}_{\phi(2n)-1}(V_{\phi(2n)}) = \tilde{C}_{\phi(2n)-1}(\varphi_{\tilde{X}_{\phi(2n)-1}}(V_{\phi(2n)}))\right]
\end{eqnarray*}
Lemma~\ref{lemme technique 1} applied to $\tilde{X}_{\phi(2n)-1}$ seen as a word of length $r_{\phi(2n)}$ on the alphabet $A^{\alpha_{2n}\ell_{\phi(2n)}}$ provides
\begin{eqnarray*}
& & \P\left[\tilde{X}_{\phi(2n)-1}(V_{\phi(2n)}) \ne f_{\phi(2n)}(C_{\phi(2n)-1},V'_{\phi(2n)})(I_{\phi(2n)})\right] \\
& &\qquad\qquad\qquad\le N^{\alpha_{2n}\ell_{\phi(2n)}}/r_{\phi(2n)}+2(N^{\alpha_{2n}\ell_{\phi(2n)}}/r_{\phi(2n)})^{1/3}.
\end{eqnarray*}

From proposition~\ref{p21}, each term converges to $0$, hence lemma~\ref{lem24} holds.
\end{proof}

\paragraph{Third step: Use of the innovations to recover $(X_{n})_{n\le 0}$} \ \\

Our aim is to show that, for every $m\le0$, $X_m$ is a function of the innovations $(V_k)_{k\le m}$.
Consider once again the events
$$A_{n}=\{ X_{\phi(2n+1)} \text{ comes from the first } \alpha_{2n}\ell_{\phi(2n)} \text{ letters of }X_{\phi(2n)} \}. $$
If the event $A_n$ occurs and if
$$X_{\phi(2n)}(I_{\phi(2n)}) = f_{\phi(2n)}(C_{\phi(2n)-1},V'_{\phi(2n)})(I_{\phi(2n)}),$$
then
$$
X_{\phi(2n+1)}=f_{\phi(2n+1)}(\cdot,V'_{\phi(2n+1)})\circ \ldots\circ f_{\phi(2n)}(\cdot,V'_{\phi(2n)})(C_{\phi(2n)-1}).
$$
Moreover, since $A_n$  depends only on $(V'_{\phi(2n)+1},\cdots,V'_{\phi(2n+1)})$, it is independent of $\{X_{\phi(2n)}(I_{\phi(2n)}) = f_{\phi(2n)}(C_{\phi(2n)-1},V'_{\phi(2n)})(I_{\phi(2n)})\}$. Thus
$$\begin{array}{r}
\P[ X_{\phi(2n+1)}=f_{\phi(2n+1)}(\cdot,V'_{\phi(2n+1)})\circ \ldots\circ f_{\phi(2n)}(\cdot,V'_{\phi(2n)})(C_{\phi(2n)-1})~|~A_n] \qquad\qquad\\
\geq \P[X_{\phi(2n)}(I_{\phi(2n)}) = f_{\phi(2n)}(C_{\phi(2n)-1},V'_{\phi(2n)})(I_{\phi(2n)})],\\
\end{array}$$
which tends to $1$ as $n$ goes to $-\infty$ by lemma~\ref{lem24}.

Some formulas below will be easier to read thanks to the introduction of the function $g_n$ which associates $X_{n+1}$ to $(X_{n},V'_{n+1})$. Namely, for every integer $n\le0$, every word $x$ in $A^{\ell_n}$ and every integer  $1\le v \le r_n$, define
\begin{eqnarray*}
g_n(x,v) & = & f_n(x,\varphi_{\tilde{x}}(v)) \text{ if }n \text{ is one of the integers }\phi(2k),\\
& = & f_n(x,v)  \text{ otherwise}.
\end{eqnarray*}

Let $X'_{\phi(2n)}$ be the word of length $\ell_{\phi(2n)}$ whose first $\alpha_{2n}\ell_{\phi(2n)}$ letters are those of the word $f_{\phi(2n)}(C_{\phi(2n)-1},V'_{\phi(2n)})$ and the others are set to 1. Then, by lemma~\ref{lem24},
$$ \P\left[X_{\phi(2n)}(I_{\phi(2n)}) = X'_{\phi(2n)}(I_{\phi(2n)})\right] \to 1 \text{ as } n \to - \infty .$$

For $n<m\le0$, call $X'_{m,n}$ the offspring of $X'_{\phi(2n)}$ at time $m$, that is,
$$
X'_{m,n}=g_{m-1}(\cdot,V'_{m})\circ g_{m-2}(\cdot,V'_{m-1})\circ\ldots\circ g_{\phi(2n)+1}(\cdot,V'_{\phi(2n)})(X'_{\phi(2n)-1}).
$$
Then,
$$
 \P\left[X_{m} \ne X'_{m,n} | A_n \right] \le \P\left[X_{\phi(2n)}(I_{\phi(2n)}) \ne X'_{\phi(2n)}(I_{\phi(2n)})| A_n \right],
$$
hence
$\P\left[X_{m} \ne X'_{m,n} | A_n \right] \to0,$ as $n$ goes to $-\infty$.

Lemma~\ref{lem2.1} below, which will be proved at the end of section~\ref{aux-split}, enables us to use
Borel-Cantelli's lemma twice.

\begin{lemme}\label{lem2.1}
Let $(a_n)_{n\ge0}$ and $(b_n)_{n\ge0}$ denote two bounded sequences of nonnegative real numbers such that the series $\displaystyle\sum_nb_n$ diverges and such that $a_n\ll b_n$. Then there exists an increasing application $\theta : \N \to \N$ such that the series
$\displaystyle\sum_{n} a_{\theta(n)}$ converges and the series $\displaystyle\sum_{n} b_{\theta(n)}$ diverges.
\end{lemme}

\paragraph{Continuation of the proof of theorem~\ref{theoC}}~\\
Since the series $\displaystyle\sum_n \P(A_n)$ diverges
and, for every fixed $m\le0$,
$$\P[X_m \ne X'_{m,n}|A_n] \to 0, \mbox{ when } n \to -\infty,$$
our last lemma applied to the sequences $(\P[A_n])_{n \le m}$ and $(\P[X_m\ne X'_{m,n};A_n])_{n\le m}$ provides a deterministic increasing application $\theta: -\n \to -\n$ such that
$$\sum_{n\le m} \P\left[X_{m} \ne X'_{m,\theta(n)} ; A_{\theta(n)} \right] < \infty \text{ and } \sum_{n\le m} \P\left[A_{\theta(n)} \right] = \infty.$$
By Borel-Cantelli's lemma, the events $\{X_{m} \ne X'_{m,\theta(n)}\} \cap A_{\theta(n)}$ occur only for a finite number of times $n$, whereas the independent events $A_{\theta(n)}$ occur infinitely often. Thus for every word $x$ in $A^{l_m}$, almost surely,
$$ \{X_{m}=x\} = \limsup_{n \to -\infty} A_{\theta(n)} \cap \{ X'_{m,\theta(n)}=x\}.
$$
This proves that $\{X_{m}=x\}$ belongs to $\F^{V'}_{m}$, hence $X_{m}$ is a function of the innovations $(V'_{m})_{m \le 0}$.
\hfill $\square$

\subsection{Proof of some auxiliary facts} \label{aux-split}

\begin{proof}[of proposition~\ref{p21}]
Assume that  $(\alpha_{n})_{n \le 0}$ and $\phi$ exist such that 1 and 2 hold. Then,
 for every $n \le 0$, $r_{\phi(n)}\ge N^{2\alpha_{n} \ell_{\phi(n)}}$ hence $\log_Nr_{\phi(n)}\ge 2\alpha_{n}\ell_{\phi(n)} $.
Therefore $$\sum_{n} \frac{\log_Nr_{n}}{\ell_{n}} \ge \sum_{n} \frac{\log_Nr_{\phi(n)}}{\ell_{\phi(n)}} \ge 2\sum_{n} \alpha_{n},
$$
and the last series diverges hence condition $\neg\Delta$ holds.

Conversely, assume that condition $\neg\Delta$ holds.
Let $\beta_{n} = \frac14\log_N(r_{n})/\ell_{n}$. Then, $r_n=N^{4\beta_{n}\ell_{n}}$ hence
$r_{n}/N^{2\beta_{n}\ell_{n}} = N^{2\beta_{n}\ell_{n}}$.

Since $\displaystyle\sum_{n\le0} \beta_{n}$ diverges and $\displaystyle\sum_{n \le 0} |n| 2^{n}$ converges, there exists an increasing application $\phi$ such that
the series $\displaystyle\sum_{n} \beta_{\phi(n)}$ diverges and such that $\beta_{\phi(n)}\ge 2^{\phi(n)} |\phi(n)|$ for every $n\le0$.

Replacing, if necessary, $\phi$ by $\varphi$ given by $\varphi(n)=\phi(n-1)$, one can ensure that the series $\displaystyle\sum_{n} \beta_{\phi(2n)}$ diverges as well.

Since $r_n\ge 2$ for every $n$, $\ell_{\phi(n)} \ge 2^{-\phi(n)}$, hence $\beta_{\phi(n)}\ge 2^{\phi(n)} |\phi(n)|$ implies that
$\beta_{\phi(n)}\ell_{\phi(n)} \ge |\phi(n)| \ge |n|.$
Hence,
$$
r_{\phi(n)}/N^{2\beta_{\phi(n)} \ell_{\phi(n)}} = N^{2\beta_{\phi(n)}\ell_{\phi(n)}}\ge N^{2 |n|}\ge1,
$$
 and this sequence converges to $+\infty$. Furthermore, defining $\alpha_n=\min(\beta_{\phi(n)},1)$, one sees that $(\alpha_n)_n$  fulfills conditions 1-2-4-5 (condition 3 being a consequence of condition 4).

We now show how to build from $(\alpha_{n})_{n}$ a sequence $(\alpha'_{n})_{n}$ such that condition 6 holds as well.

Recall that from the construction above, $r_n=N^{4\beta_n\ell_n}$ and $\beta_{\phi(n)}\ell_{\phi(n)}\ge|\phi(n)|$, hence
 $r_{\phi(n)} \ge N^{4|\phi(n)|} \ge 2^{-4\phi(n)}$. Also $\alpha_n \ge 2^{\phi(n)}|\phi(n)|$.

These inequalities implies that
the ratio $\varrho_n=\alpha_{n} \ell_{\phi(n)}/\ell_{\phi(n+1)} $ is such that
\begin{eqnarray*}
\varrho_n&=& \alpha_{n}\left(\prod_{k=\phi(n)+1}^{\phi(n+1)-1} r_{k}\right) r_{\phi(n+1)} \\
&\ge&2^{\phi(n)}|\phi(n)|~ 2^{\phi(n+1)-\phi(n)-1} 2^{-4\phi(n+1)} \\
&=&2^{-3\phi(n+1)-1} |\phi(n)|.
\end{eqnarray*}
Since $\phi(n+1) \le -1$ and  $|\phi(n)|\ge2$ for every $n\le-2$, this shows that $\varrho_n\ge8$ for every $n\le-2$.

Thus, $8\varrho_n/9\le\lfloor\varrho_n \rfloor\le\varrho_n$ and the sequence $(\alpha'_{n})_{n}$
defined by
$$\alpha'_{n} = \frac{\lfloor\alpha_{n}\ell_{\phi(n)}/\ell_{\phi(n+1)} \rfloor}{\ell_{\phi(n)}/\ell_{\phi(n+1)} },
$$
is such that $\frac89 \alpha_n \le \alpha'_n \le \alpha_n$ for every $n\le-2$. Therefore it fulfills the conditions already satisfied by the sequence $(\alpha_{n})_{n\le 0}$ and the additional  condition that $\alpha'_n\ell_{\phi(n)}/\ell_{\phi(n+1)}$ is an integer for every $n \le 0$. \hfill $\square$ \end{proof}

\begin{proof}[of lemma~\ref{lem2.1}]
Call $B$ any finite upper bound of the sequence $(b_n)_{n \in \n}$. Since $a_n\ll b_n$, for any positive integer $k$, there exists an integer $N_k$ such that for every $n \ge N_k$, $a_n \le b_n 2^{-k}$.
We now define a sequence of disjoint intervals of integers $J_k=\{i_k,\ldots,j_k\}$, as follows. Set $j_{-1}=-1$ and let $k\ge0$.

Once $j_{k-1}$ is defined, let $i_k=\max\{j_{k-1}+1,N_k\}$.
Since the series $\displaystyle\sum_{n} b_n$ diverges and $0\le b_n \le B$ for every $n \in \n$, one can choose an integer $j_k\ge i_k$ such that $\displaystyle B \le \sum_{n=i_k}^{j_k} b_n <2B$ and let $I_k=\{i_k,\ldots,j_k\}$. Note that $\displaystyle\sum_{n \in J_k} b_n \ge B$ and $\displaystyle\sum_{n \in J_k} a_n \le 2B/2^k$.

Calling $Q$ the set $\displaystyle\bigcup_{k \in \n} I_k$, one gets
$$ \sum_{{n \in Q}} a_{n} = \sum_{{k \in \n}} \sum_{{n \in J_k}} a_n \le 2B \sum_{{k \in \n}} 1/2^k =4B,$$
and
$$ \sum_{{n \in Q}} b_{n} = \sum_{{k \in \n}} \sum_{{n \in J_k}} b_n \ge \sum_{{k \in \n}} B = \infty,
$$
which completes the proof of the lemma. \hfill $\square$ \end{proof}

\section{Proof of non standardness under $\Delta$}
\label{S3-split}

The non standardness of the split-word process was established by Laurent in his thesis  \cite{laurent2004ftd} and a similar result was obtained by Vershik in \cite{vershik1995tds} in the context of decreasing sequences of measurable partitions. In this section we give a simplified presentation of Laurent's proof.

The proof of this result involves a subtle notion on filtrations, which is standardness.
The notion of standardness was first introduced by Vershik for decreasing measurable partitions. This notion has been adapted to continuous time filtrations by Tsirelson~\cite{tsirelson1997triple}, it has been formulated by Dubins, Feldman, Smorodinsky and Tsirelson~\cite{dubins1996decreasing} for continuous time filtrations and by Emery and Schachermayer~\cite{emery2001vershik} for discrete time filtrations.


Many necessary and sufficient conditions for standardness have been established, for instance Vershik's
self-joining criterion and various notions of cosiness. All these criterions are based on coupling methods. Checking them in specific cases is often a technical task. Yet, these criterions are the key tool to solve some difficult problems. For example, Tsirelson defines and uses a notion of cosiness to prove that the filtration of Walsh's Brownian motion is not Brownian since it is non standard.

By definition, a filtration $\F=(\F_n)_{n\le0}$ indexed by $n \le 0$ is \textit{standard\/}
if, modulo an enlargement of the probability space, one can immerse $\F$ in a filtration generated by an i.i.d.\ process.
Recall that a filtration $\F=(\F_n)_{n \le 0}$ is \textit{immersed\/} in a filtration $\G=(\G_n)_{n \le 0}$ if, for every $n\le0$, $\F_n\subset\G_n$ and $\F_{n}$ and $\G_{n-1}$ are independent conditionally on $\F_{n-1}$. Roughly speaking, this means that $\G_{n-1}$ gives no further information on $\F_{n}$ than $\F_{n-1}$ does. Equivalently, $\F$ is immersed in $\G$ if every $\F$-martingale is a $\G$-martingale.

Laurent negates the so-called I-cosiness property to prove that under $\Delta$, the filtration is non standard and therefore non of product type.

We follow the method that Smorodinsky \cite{smorodinsky1998pns} used to prove the non-existence of a ``generating parametrization'' in the case where $r_n = 2$ for every $n\le 0$. This method still works in the general case and provides the non standard behaviour of the filtration.

The purpose of this section is to show that if the sequence $(r_{n})_{n \le 0}$ is $\Delta$, then the filtration $\mathcal{F}^{(X,V)}$ is non standard.
Note that
$\Delta$ holds for every bounded sequence $(r_{n})_{n}$.

\subsection{Preliminary notions}
\label{S3.1-split}

We first recall the notion of I-cosiness, due to \'Emery and Schachermayer \cite{emery2001vershik}.

\begin{definition}[Immersion and co-immersion]
Let $\mathcal{F}=(\mathcal{F}_n)_{n \le 0}$ and $\mathcal{G}=(\mathcal{G}_n)_{n \le 0}$ denote two filtrations defined on the same probability space.
\begin{itemize}
\item  $\mathcal{F}$ is immersed in $\mathcal{G}$ if every martingale in $\mathcal{F}$ is a martingale in $\mathcal{G}$.
\item  $\mathcal{F}$ and $\mathcal{G}$ are co-immersed if $\mathcal{F}$ and $\mathcal{G}$ are both immersed in $\mathcal{F} \vee \mathcal{G}$.
\end{itemize}
\end{definition}

\begin{definition}[I-cosiness]
Let $\mathcal{F}$ be a filtration on a probability space $(\Omega,\mathcal{A},\P)$. One says that $\mathcal{F}$ satisfies the I-cosiness criterion if for every random variable $Y$ measurable for $\mathcal{F}_0$ with values in a finite set and for every real $\delta > 0$, there exists a probability space $(\overline{\Omega},\overline{\mathcal{A}},\overline{\P})$ and two filtrations $\mathcal{F'}$  and $\mathcal{F''}$ on $(\overline{\Omega},\overline{\mathcal{A}},\overline{\P})$, such that the following properties hold.
\begin{itemize}
\item The filtrations $\mathcal{F}'$ and $\mathcal{F}''$ are both isomorphic to $\mathcal{F}$.
\item The filtrations $\mathcal{F}'$ and $\mathcal{F}''$ are co-immersed.
\item  There exists an integer $n_0$ such that $\mathcal{F}'_{n_0}$ and $\mathcal{F}''_{n_0}$ are independent.
\item The copies $Y'$ and $Y''$ of $Y$ by the isomorphisms of the first condition, verify $\overline{\P}[Y' \ne Y'']<\delta$.
\end{itemize}
\end{definition}

The proof of the non standardness of the filtration $\mathcal{F}^{(X,V)}$ uses the easy part of the equivalence between I-cosiness and standardness.

\begin{theoreme}[Corollary~5~\cite{emery2001vershik}]
A filtration is standard if and only if it satisfies the I-cosiness criterion and is essentially separable.
\end{theoreme}

To prove the non standardness of $\mathcal{F}^{(X,V)}$, it is therefore sufficient to show that $\mathcal{F}^{(X,V)}$ does not satisfies the I-cosiness criterion.
The tools of the proof are introduced just below.

Let $n_0$ be a negative integer which will be fixed later.

\begin{definition}[Definition of $\mathrm{Aut}_n$]
The intervals of integers $\{1,\ldots,\ell_k\}$,\\ $\{\ell_k+1,\ldots,2 \ell_k\},\ldots,\{\ell_n-\ell_k+1,\ldots,\ell_n\}$ are called blocks of $\{1,\ldots,\ell_n\}$ of length $\ell_k$.
Every permutation of the $\ell_{n}/\ell_{n_{0}}$ blocks of length $\ell_{n_{0}}$ which induces for every $k \in \{n,\ldots,n_{0}\}$ a bijection between the blocks of length
$\ell_{k}$ is called an automorphism of $\{1,\ldots,\ell_{n}\}$ adapted to $\{r_{n+1},...,r_{n_0}\}$. One denotes by $\rm{Aut}_n$ the set of those permutations.
\end{definition}

One can enumerate the automorphisms adapted to $(r_{n})_{n \le 0}$, by induction, as follows.
By definition, an automorphism $a$ in $\mathrm{Aut}_{n-1}$ is built from a permutation $\sigma$ of $\{1,\ldots,r_n\}$ and from $r_n$ automorphisms $(a_k)_{1\le k\le r_n}$ in $\mathrm{Aut}_{n}$. One gets $a$ from $\sigma$ and $(a_k)_{1\le k\le r_n}$ by setting, for every $1\le j\le r_n$ and every $1\le k\le \ell_n$,
$$a((j-1)\ell_{n}+k) = (\sigma(j)-1)\ell_{n}+a_{j}(k).
$$
Therefore $\# (\mathrm{Aut}_{n-1})= \#(\mathfrak{S}_{r_{n}})(\#(\mathrm{Aut}_{n}))^{r_{n}}=r_{n}!(\#(\mathrm{Aut}_{n}))^{r_{n}}$.
By induction
$$\#(\mathrm{Aut}_{n}) = \prod^{n_{0}}_{k=n+1}(r_{k}!)^{r_{n+1}\ldots r_{k-1}} = \prod^{n_{0}}_{k=n+1}(r_{k}!)^{\ell_{n}/\ell_{k-1}}.$$
Note that
$$\#(\mathrm{Aut}_{n})= \exp(\ell_{n}S_{n})\ \text{ where } \displaystyle \ S_{n}=\sum_{k=n+1}^{n_{0}}\frac{\ln r_k!}{\ell_{k-1}}.
$$
We now define a semi-metrics based on the Hamming distance.

\begin{definition}[Semi-metrics $e_n$ on $A^{\ell_{n}}$]
Recall that the Hamming distance on $A^{\ell_{n}}$ is defined by
 $$d^{H}_{n}(x,x')=\#\{ k \in \{1,\ldots,\ell_{n}\} ~:~  x(k) \ne x'(k) \}.
 $$
One defines an action of the group $\mathrm{Aut}_{n}$ on $A^{\ell_{n}}$,
seen as $(A^{\ell_{n_0}})^{\ell_n/\ell_{n_0}}$,
by
$$
a \cdot x = x \circ a^{-1},
$$
and a semi-metrics $e_{n}$ on $A^{\ell_{n}}$, by
$$
e_{n}(x,x')=\frac{1}{\ell_{n}} \min\{d^{H}_{n} (a\cdot x, x')\,;\,a \in \mathrm{Aut}_{n}\}.
$$
\end{definition}

The quantity $e_{n}(x,x')$ is the smallest proportion of letters which are different between the words $x'$ and $x \circ a^{-1}$ as $a$ goes through $\mathrm{Aut}_{n}$.
Our next result is a recursion relation, useful to compute $e_n$.

\begin{lemme}\label{lem3.1} For every $n\le0$,
$x=(w_{1},\ldots,w_{r_{n}})$ and $x'=(w'_{1},\ldots,w'_{r_{n}})$
where $w_{i}$ and $w'_j$ belong to $A^{\ell_{n}}$,
$$ e_{n-1}(x,x')=\min_{\sigma \in \mathfrak{S}_{r_{n-1}}} \left(\frac{1}{r_n} \sum_{j=1}^{r_{n}} e_{n}(w_{j},w'_{\sigma(j)})\right).$$
\end{lemme}

\begin{proof}[of lemma~\ref{lem3.1}]

Using the decomposition of every automorphism $a$ of $\mathrm{Aut}_{n-1}$ into a permutation $\sigma$ of $\{1,\ldots,r_n\}$ and $r_n$ automorphisms $a_k$, $1\le k\le r_n$ in $\mathrm{Aut}_{n}$, and the additivity of the restricted Hamming distance,
one sees that $\ell_{n-1} e_{n-1}(x,x')$ is the minimum over these $\sigma$ and $a_k$ of
the sums
$$
\sum_{j=1}^{r_n}d_{n}^{H}(a_{j} \cdot w_{j}, w'_{\sigma(j)}),
$$
hence
$$
\ell_{n-1} e_{n-1}(x,x')
=
\min_{\sigma \in \mathfrak{S}_{r_{n}}} \left(\ell_{n}\sum_{j=1}^{r_n}e_{n}(w_{j},w'_{\sigma(j)}) \right).
$$
\hfill $\square$
\end{proof}


To prove that the filtration $\mathcal{F}^{(X,V)}$ is non standard under the assumption $\Delta$, by denying the I-cosiness criterion, one considers $X'$ and $X''$ two copies of $(X_n)_{n \le 0}$ such that:
\begin{itemize}
\item The associated filtrations $\mathcal{F}'=\mathcal{F}^{(X')}$ and $\mathcal{F}''=\mathcal{F}^{(X'')}$ are co-immersed;
\item There exists an integer $m<n_0$ such that $(X'_k)_{k\le m}$ and $(X''_k)_{k\le m}$ are independent.
\end{itemize}
Our proof of the non standardness of $\F^{(X,V)}$ includes three steps:
\begin{itemize}
\item We prove the inequality $\P[X'_{n_0} \ne X''_{n_0}] \ge \E[e_n(X'_n,X''_n)]$.
\item We bound below $\E [e_{n}(X'_n,X''_n)]$ when $X'_n$ and $X''_n$ are independent.
\item We negate the I-cosiness criterion.
\end{itemize}

\subsection{Proof of the inequality $\P[X'_{n_0} \ne X''_{n_0}] \ge \E[e_n(X'_n,X''_n)]$ for $n\le n_0$}\label{S3.2-split}


For $n\le n_0$ one denotes by
$$
M_n = \P[X'_{n_0} \ne X''_{n_0}|{\cal F'}_n \vee {\cal F''}_n],
\qquad
L_n = e_n(X'_n,X''_n).
$$
By construction, $(M_n)_{n \le n_0}$ is a martingale. Loosely speaking, this martingale measures the influence of the past before time $n$ on the word at time $n_0$. The key step is to prove that $(M_n)_{n \le n_0}$ is bounded below by $(L_n)_{n \le n_0}$.

Let us prove that $(L_n)_{n \le n_0}$ is a sub-martingale in the filtration ${\cal F'} \vee {\cal F''}$.
Let $n \le n_0$. Using the co-immersion of the filtrations ${\cal F'}$ and ${\cal
 F''}$, the conditional law of $X'_n$ given ${\cal F'}_{n-1} \vee
{\cal F''}_{n-1}$ is uniform on the $r_n$ sub-words of $X'_{n-1}$ of length $\ell_n$:
$${\cal L}(X'_n|{\cal F'}_{n-1} \vee
{\cal F''}_{n-1}) = \frac{1}{r_n} \sum_{i=1}^{r_n}\ \delta_{f_n(X'_{n-1},i)},$$
where $f_n(x,v)$ denotes the $v$-th sub-word (of length $\ell_n$) of $x$, for $x\in A^{\ell_{n-1}}$ and $1\le v \le r_n$ (as in the proof of theorem~\ref{theoB}). Furthermore,
$${\cal L}(X''_n|{\cal F'}_{n-1} \vee {\cal F''}_{n-1}) = \frac{1}{r_n} \sum_{j=1}^{r_n}\ \delta_{f_n(X''_{n-1},j)}.$$
Therefore, the conditional law of $(X'_n,X''_n)$ given ${\cal
 F'}_{n-1} \vee {\cal F''}_{n-1}$ can be written as follows
$${\cal L}((X'_n,X''_n)|{\cal F'}_{n-1} \vee
{\cal F''}_{n-1}) = \frac{1}{r_n} \sum_{1 \le i,j \le r_n} C_{i,j}\ \delta_{\big(f_n(X'_{n-1},i),f_n(X''_{n-1},j) \big)},$$
where $(C_{i,j})_{1 \le i,j \le r_n}$ is a bistochastic matrix
 measurable for $\F'_{n-1} \vee \F''_{n-1}$. In particular,
$$\E [L_n|{\cal F'}_{n-1} \vee
{\cal F''}_{n-1}] = \frac{1}{r_n} \sum_{1 \le i,j \le r_n} C_{i,j} \ e_n \big(f_n(X'_{n-1},i),f_n(X''_{n-1},j) \big).$$
This quantity is the image of the matrix $(C_{i,j})_{1 \le i,j \le
 r_n}$ by a linear form. Since bistochastic matrices belong to the convex hull of the permutation matrices, one gets
$$
\E [L_n|{\cal F'}_{n-1} \vee {\cal F''}_{n-1}] \ge \frac{1}{r_n} \inf_{\sigma \in \mathfrak{S}_{r_n}} \sum_{i=1}^{r_n}  e_n \big(f_n(X'_{n-1},i),f_n(X''_{n-1},\sigma(i)) \big),
$$
hence
$$
\E [L_n|{\cal F'}_{n-1} \vee {\cal F''}_{n-1}] \ge e_{n-1} (X'_{n-1},X''_{n-1}) = L_{n-1},
$$
thanks to the recursion relation verified by the semi-metrics
$e_n$. This shows that $(L_n)_{n \le n_0}$ is a sub-martingale
in the filtration ${\cal F'} \vee {\cal F''}$.

Since $M_{n_0} = \mathbf{1}_{\{X'_{n_0} \ne X''_{n_0}\}} \ge L_{n_0}$, one gets, for every $n \le n_0$
$$
M_n = \E [M_{n_0}|{\cal F'}_n \vee {\cal F''}_n] \ge \E [L_{n_0}|{\cal F'}_n \vee {\cal F''}_n] \ge L_n,
$$
yielding the inequality $\P[X'_{n_0} \ne X''_{n_0}] \ge \E [e_n(X'_n,X''_n)]$ by taking the expectations.

This inequality is going to be used to prove that the probability
$\P[X'_{n_0} \ne X''_{n_0}]$ can not be made as small as one wishes if the processes $X'$ and $X''$ are  independent until a given time $n$.

\subsection{Bounding below $\E [e_{n}(X'_n,X''_n)]$ when $X'_n$ and $X''_n$ are independent}
\label{S3.3-split}
The purpose of this subsection is to prove the next inequality:
\begin{lemme} \label{SL.11.4.1}\label{lem3.2}
Let $(X'_{k})_{k\le 0}$ and $(X''_{k})_{k\le 0}$ be two copies of the split-word process that are independent until time $n$.
Then, for every $\alpha>0$ and $n \le n_0$,
$$ \P \left[ e_{n}(X'_{n},X''_{n}) \le 1-N^{-1}-\alpha \right] \le \exp\left(\ell_{n}(S_{n}-2\alpha^2)\right), \text{ where } S_{n}=\sum_{k=n+1}^{n_{0}}\frac{\ln(r_{k}!)}{\ell_{k-1}} $$
\end{lemme}

The proof is based on Hoeffding's large deviations inequality, see \cite{shiryaev1996png} for a proof of this inequality.

\begin{lemme}[Hoeffding] \label{Hoeffding} \label{lem3.3}
Let $q$ in $]0,1[$ and $(\epsilon_k)_{1\le k\le n}$ denote $n$ independent Bernoulli random variables of parameter $q$, and $Z_{n}$ their average. Then, for every real $\alpha > 0$,
$$
\P \left[ Z_{n} \le q-\alpha\right] \le \exp(-2n\alpha^2).
$$
\end{lemme}

\begin{proof}[of lemma~\ref{lem3.2}]
One applies lemma~\ref{Hoeffding} to the variables
$$\epsilon_{i}= \mathbf{1}_{\{X'_{n}(i) \ne
X''_{n}(i)\}},\qquad 1\le i \le \ell_{n}.
$$
Since $X'_n$ and $X''_n$ are independent and uniform on $A^{\ell_n}$, the $\epsilon_i$ are independent and Bernoulli with parameter $q_N=1-1/N$. Thus, denoting by $d^H_n$ the Hamming distance on $A^{\ell_n}$,
$$
\P \left[ \frac{1}{\ell_{n}} d^H_{n} (X'_{n},X''_{n}) \le q_N-\alpha \right] \le \exp(-2 \ell_{n} \alpha^2).
$$
For every $a$ in $\mathrm{Aut}_{n}$, $a \cdot X''$ has the same law as $X''$ and is independent of $X'$. Since, by definition $\displaystyle e_{n}(X'_{n},X''_{n})=\min\{d^H_{n}(X'_{n},a \cdot X''_{n})\,;\,a \in \mathrm{Aut}_{n}\}$, one gets
$$
 \left[ e_{n}(X'_{n},X''_{n}) \le q_N-\alpha \right]=\left[ \exists a \in \mathrm{Aut}_{n},~ \frac{1}{\ell_{n}} d^H_{n}(X'_{n},a \cdot X''_{n}) \le q_N-\alpha \right],
$$
hence
 \begin{eqnarray*}
\P \left[ e_{n}(X'_{n},X''_{n}) \le q_N-\alpha \right] &\le& \sum_{{a \in \mathrm{Aut}_{n}}} \P \left[ \frac{1}{\ell_{n}} d^H_{n}(X'_{n},a \cdot X''_{n}) \le q_N-\alpha \right] \\
&=& \sharp (\mathrm{Aut}_{n}) ~\P \left[\frac{1}{\ell_{n}}  d^H_{n}(X'_{n},X''_{n}) \le q_N-\alpha \right] \\
&\le& \exp(\ell_{n} S_{n}) \exp(-2\ell_{n}\alpha^2),
\end{eqnarray*}
since $\sharp (\mathrm{Aut}_{n})=\exp(\ell_{n} S_{n})$.
\hfill $\square$
\end{proof}

\paragraph{Choice of the integer $n_0$}
Under assumption $\Delta$, by remark 1 after theorem~\ref{theoC}, one can choose $n_0 \le 0$ such that
$$ S_{-\infty}<2 \left(1-1/N \right)^2, \text{ where } S_{-\infty}=\sum_{k=-\infty}^{n_0} \frac{\ln(r_{k}!)}{\ell_{k-1}}.
$$
Once $n_0$ is fixed, choose a real $\alpha$ such that
$\displaystyle \sqrt{S_{-\infty}/2}<\alpha<1-1/N.$

Since
$S_n-2\alpha^2\le S_{-\infty}-2\alpha^2<0$ and $\ell_{n} \ge \ell_{n_{0}}$ for every $n \le n_0$, lemma~\ref{SL.11.4.1} yields
\begin{eqnarray*}
\P \left[ e_{n}(X'_{n},X''_{n}) \le 1-1/N-\alpha \right] &\le& \exp\left(\ell_{n}(S_n-2\alpha^2)\right) \\
&\le& \exp\left(\ell_{n_0}(S_{-\infty}-2\alpha^2)\right)=\beta,
\end{eqnarray*}
with $\beta<1$. Hence, $ \E [e_{n}(X'_{n},X''_{n})] \ge (1-\beta)\left(1-1/N-\alpha\right)$, which is positive.

\subsection{Negation of the I-cosiness criterion}\label{S3.4-split}
If $(X',V')$ and $(X'',V'')$ are two copies of the split-word processus whose filtrations are co-immersed and independent until time $n\le n_0$,
$$
\P[X'_{n_0} \ne X''_{n_0}]\ge \E [e_{n}(X'_{n},X''_{n})] \ge (1-\beta)(1-1/N-\alpha),
$$
and this lower bound is positive,
contradicting the I-cosiness criterion. Thus the filtration of the split-word process is non standard. $\hfill \square$


\begin{thebibliography}{xx}

\bibitem{tsirelson1997triple}
B. Tsirelson \textsl{Triple points: from non-{B}rownian filtrations to harmonic measures},
Geometric and Functional Analysis,
1997


\bibitem{laurent2009vershikian}
S. Laurent, \textsl{On {V}ershikian and {I}-cosy random variables and filtrations},
IAP Network, TR09018, 2009

\bibitem{vershik1973four},
A.M. Vershik, \textsl{Four definitions of the scale of an automorphism},
Functional Analysis and Its Applications, 1973

\bibitem{dubins1996decreasing},
L. Dubins and J. Feldman and M. Smorodinsky and B. Tsirelson,
\textsl{Decreasing sequences of $\sigma$-fields and a measure change for {B}rownian motion},
The Annals of Probability, 1996


\bibitem{emery2001vershik},
M. {\'E}mery and W. Schachermayer,
\textsl{On {V}ershik's standardness criterion and {T}sirelson's notion of cosiness},
S\'eminaire de {P}robabilit\'es, XXXV,
2001

\bibitem{smorodinsky1998pns},
M. Smorodinsky,
\textsl{Processes with no standard extension},
Israel Journal of Mathematics, 1998

\bibitem{gorbulsky1999one},
A. Gorbulsky,
	\textsl{About one property of entropy of decreasing sequence of measurable partitions},
Zapiski Nauchnykh Seminarov POMI, 1999

\bibitem{laurent2004ftd},
S. Laurent, \textsl{Ph{D} thesis : Filtrations \`a temps discret n\'egatif},
Universit\'e Louis Pasteur, Institut de Recherche
 Math\'ematique Avanc\'ee, {S}trasbourg, France, 2004

\bibitem{shiryaev1996png},
A.N. Shiryaev, \textsl{Probability},
Graduate Texts in Mathematics, 1996

\bibitem{vershik1995tds},
A.M. Vershik, \textsl{Theory of decreasing sequences of measurable partitions},
Saint Petersburg Mathematical Journal, 1995

\bibitem{laurent2009pre},
S. Laurent, \textsl{On standardness and {I}-cosiness}, IAP Network, 2009

\bibitem{vershik1970},
A.M. Vershik, \textsl{Decreasing sequences of measurable partitions, and their applications.},
Soviet mathematics Doklady, 1970

\bibitem{heicklen},
D. Heicklen, \textsl{Bernoullis are standard when entropy is not an obstruction},
Israel Journal of Mathematics, 1998


\end{thebibliography}

\end{document}